\documentclass[11pt]{amsart}

\usepackage{pinlabel, graphicx}
\usepackage[all,knot,poly]{xy} 
\usepackage{hyperref}
\usepackage[usenames,dvipsnames]{color}
\usepackage{enumerate}

\newcommand{\leg}{\ensuremath{\Lambda}}
\DeclareMathOperator{\rk}{rank}

\newcommand{\ccc}{\ensuremath{\mathbb{C}}}
\newcommand{\rr}{\ensuremath{\mathbb{R}}}
\newcommand{\zz}{\ensuremath{\mathbb{Z}}}

\theoremstyle{plain}
\newtheorem{thm}{Theorem}[section]
\newtheorem{cor}[thm]{Corollary}
\newtheorem{lem}[thm]{Lemma}

\newtheorem{prop}[thm]{Proposition}

\theoremstyle{definition}
\newtheorem{defn}[thm]{Definition}

\newtheorem{ques}[thm]{Question}

\theoremstyle{remark}
\newtheorem{rem}[thm]{Remark}
\newtheorem{ex}[thm]{Example}

\numberwithin{equation}{section}

\def\dfn#1{{\textbf {#1}}}

\begin{document}

\title[Topologically Distinct Fillings]{Topologically Distinct Lagrangian and Symplectic Fillings}

\date{\today}

\author[C. Cao]{Chang Cao} \address{University of Pennsylvania, Philadelphia, PA 19104} \email{ccao@sas.upenn.edu}

\author[N. Gallup]{Nathaniel Gallup} \address{University of California, Davis,
CA 95616} \email{n.p.gallup@gmail.com}

\author[K. Hayden]{Kyle Hayden} \address{Boston College, Chestnut Hill, MA 02467} \email{kyle.hayden@bc.edu}

\author[J. Sabloff]{Joshua M. Sabloff} \address{Haverford College,
Haverford, PA 19041} \email{jsabloff@haverford.edu} 

\begin{abstract} We construct infinitely many Legendrian links in the standard contact $\rr^3$ with arbitrarily many topologically distinct Lagrangian fillings.  The construction is used to find  links in $S^3$ that bound topologically distinct pieces of algebraic curves in $B^4 \subset \ccc^2$, is applied to find contact $3$-manifolds with topologically distinct symplectic fillings, and is generalized to higher dimensions.
\end{abstract}

\maketitle


\section{Introduction}
\label{sec:intro}

Under certain conditions, the smooth topology of a symplectic manifold is governed by the contact topology of its boundary.  The first realization of this principle was Eliashberg's proof that any symplectically aspherical filling of the standard contact $S^3$ is diffeomorphic to $B^4$ \cite{yasha:filling}; see also \cite{mcduff:contact-type, ov:fillings} for similar results in higher dimensions and \cite{mcduff:rational, oo:simple-elliptic, oo:simple-singularity, stipsicz:T3-filling, wendl:strong-fillings} in low dimensions. In the relative setting, the original principle was first realized by a theorem of Chantraine \cite{chantraine}: if $L$ is an orientable Lagrangian filling of an oriented Legendrian link in the boundary of a Stein surface, then
\begin{equation} \label{eq:tb-euler}
  tb(\leg) = -\chi(L).
\end{equation}
In particular, if $\leg$ is connected --- i.e.\ a knot --- then the topology of a Lagrangian filling of $\leg$ is completely determined by $tb(\leg)$.

While recent work has revealed a wide variety of failures of this principle in the case of symplectic manifolds --- there are contact $3$-manifolds with arbitrarily many non-homeomorphic fillings \cite{lisca:lens}, infinitely many non-homeomorphic fillings \cite{oo:infinite-fillings,os:infinite-fillings,smith:infinite-fillings}, and even infinitely many homeomorphic but not diffeomorphic fillings \cite{aems} --- little is known in the relative setting.  In the somewhat softer setting of symplectic fillings of transverse knots in $S^3$, Geng \cite{geng:2-surfaces} notes that Auroux, Kulikov, and Shevchishin \cite{aks:hurwitz} found an example of a transverse braid with two topologically distinct symplectic fillings.   Motivated by Chantraine's result and by a question of Boileau and Fourrier \cite[Question 1]{boileau-fourrier} in an analogous situation in the realm of algebraic curves in $\ccc^{2}$, we  ask:

\begin{ques} \label{ques:main}
  Is the topology of a Lagrangian filling of an oriented Legendrian \emph{link} in $S^3$ completely determined by its Thurston-Bennequin invariant?
\end{ques}

In contrast to the case of Legendrian \emph{knots}, the main theorem of this paper implies that the answer to this question is a resounding \emph{no}. We call a Legendrian link with topologically distinct Lagrangian fillings \dfn{polyfillable}.

\begin{thm} \label{thm:leg-polyfillable}
  For every integer $N>1$, there exists a Legendrian link $\leg \subset (S^3, \xi_0)$ with $p(N)$ non-homeomorphic exact orientable Lagrangian fillings.
\end{thm}

In fact, we can construct links that have an enormous amount of flexibility in their polyfillability, and the theorem above is a simple corollary of one of several more refined theorems that we state and prove in Section~\ref{sec:leg-proofs}.

Though the theorem above is interesting in and of itself, we may also apply it in three different directions.  
First, Theorem~\ref{thm:leg-polyfillable} leads to a solution of Boileau and Fourrier's original question:

\begin{thm} \label{thm:alg-polyfillable}
For every  integer $N>1$, there exists a link $K \subset S^{3}$ with $p(N)$ non-singular complex algebraic curves that intersect $S^3$ transversally in $K$ with the property that the intersections of the curves with $B^4$ are pairwise non-homeomorphic.
\end{thm}

We note that the method by which we connect Theorem~\ref{thm:leg-polyfillable} to Theorem~\ref{thm:alg-polyfillable} shows that Auroux, Kulikov, and Shevchishin's pair of transverse fillings \cite{aks:hurwitz} also yields a negative answer to Boileau and Fourrier's question.

Second, Theorem~\ref{thm:leg-polyfillable} may also be applied to the construction of multiple strong symplectic fillings of contact manifolds.  We will show below that we may perturb the Lagrangian fillings from Theorem~\ref{thm:leg-polyfillable} into symplectic fillings of transverse links.  Taking double branched covers over the symplectic fillings then yields the following:

\begin{thm} \label{thm:cont-polyfillable}
  For every positive integer $N$, there exists a contact $3$-manifold $(Y, \alpha)$ with strong symplectic fillings $(X_{1}, \omega_{1}), \ldots, (X_{N},\omega_{N})$ that have the same Euler characteristic, but for which
  $$\operatorname{rank} H_3(X_{k}, Y) = k-1.$$
In particular, the fillings $X_{k}$ do not differ solely by a symplectic blow-up.
\end{thm}

We note that this result is not as spectacular as some of those mentioned in the discussion of symplectic fillings, above, but we believe that the branched cover technique is novel.

Finally, we may extend the theorem to higher dimensions while dropping the requirement that the Legendrian is disconnected:

\begin{thm} \label{thm:high-d-leg-polyfillable}
For every $n > 1$ and every  integer $N>1$, there exists a connected Legendrian submanifold $\leg \subset (S^{2n+1}, \xi_0)$ with $N$  non-homeomorphic connected exact orientable Lagrangian fillings.
\end{thm}

As we shall discuss  in Section~\ref{ssec:high-d}, this construction implies that there are Legendrian submanifolds with arbitrarily large sets of generating family homology or linearized Legendrian contact homology invariants, generalizing the result of Melvin and Shrestha \cite{melvin-shrestha} for knots in the standard contact $\rr^3$.

The remainder of the paper is organized as follows.  We sketch the relevant background material on Lagrangian and symplectic fillings in Section~\ref{sec:background}.  The proof of Theorem~\ref{thm:leg-polyfillable} appears in Section~\ref{sec:leg-proofs}, while its application to the proofs of Theorems~ \ref{thm:alg-polyfillable},  \ref{thm:cont-polyfillable}, and \ref{thm:high-d-leg-polyfillable}  appears in Section~\ref{sec:applications}.  

\subsection*{Acknowledgements}

We thank Matt Hedden for bringing Boileau and Fourrier's question to our attention and for explaining the method for deriving Theorem~\ref{thm:alg-polyfillable} from Theorem~\ref{thm:leg-polyfillable}.  We thank Curtis Greene for explaining the proper viewpoint on compositions and partial orders that we use in stating Theorem~\ref{thm:general case}. Finally, we  gratefully acknowledge support from NSF grant DMS-0909273.

\section{Lagrangian and Symplectic Fillings}
\label{sec:background}

We begin this section by setting down some of the basic language of Legendrian knot theory, Lagrangian fillings, and symplectic fillings of contact manifolds. We assume a basic familiarity with Legendrian knot theory and contact topology; for further background, see the introductory articles of Etnyre \cite{etnyre:intro, etnyre:knot-intro} or Geiges' text \cite{geiges:intro}.  With the appropriate language established, we review the geometric tools for constructing Lagrangian surfaces that underlie the main theorem.

\subsection{Lagrangian and Symplectic Fillings}
\label{ssec:fill-background}

We begin by considering Legendrian links in a  contact manifold $(Y, \alpha)$ and their Lagrangian fillings in the symplectization $(\rr \times Y, d(e^t \alpha))$.  A Lagrangian submanifold $L \subset \rr \times Y$ is a \dfn{Lagrangian cobordism} between Legendrian links $\leg_-, \leg_+ \subset Y$ if there exist real numbers $T_- < T_+$ so that 
\begin{align*}
	L \cap (-\infty, T_-] \times Y &= (-\infty, T_-] \times \leg_-,\\
	L \cap [T_+, \infty) \times Y &= [T_+,\infty) \times \leg_+.
\end{align*}
We write $\leg_- \prec_L \leg_+$ to denote this situation. In the language of \cite{chantraine:collar}, this definition requires our Lagrangian cobordisms to be collared.  A Lagrangian cobordism from $\emptyset$ to $\leg$ is a \dfn{Lagrangian filling} of $\leg$.  In this paper, we will be primarily concerned with the standard contact structures on $\rr^3$ and $S^3$.

A symplectic manifold $(X,\omega)$ is a \dfn{strong symplectic filling} of a contact manifold $(Y,\alpha)$ if $Y = \partial X$ and there exists a non-vanishing Liouville vector field $W$ for $\omega$ that is transverse to $Y$ and so that $\alpha = \iota_W \omega|_Y$.  In lieu of the Liouville vector field, we could also have specified that there is a collar neighborhood of $Y$ in $X$ that is symplectomorphic to $((-a,0] \times Y, d(e^t \alpha))$.  The canonical example of a strong symplectic filling is the filling of the standard contact structure on $S^3$ by the standard symplectic structure on $B^4$.  Notice that just as any Legendrian link $\leg$ in the standard contact $\rr^3$ can be thought of as living in the standard contact $S^3$, so any Lagrangian filling of $\leg$ immediately yields a properly embedded collared Lagrangian surface in $B^4$. We may thus make our constructions in the symplectization of $\rr^3$ but state the resulting theorems in the filling $B^4$ of $S^3$.

\subsection{Constructions of Lagrangian Fillings}
\label{ssec:constr-background}

Our main geometric tool for constructing Lagrangian fillings, embodied by the theorem below, was first announced by Ekholm, Honda, and K\'alm\'an \cite{ehk:leg-knot-lagr-cob}.  The first part of the theorem was proven by Chantraine \cite{chantraine}, and the last two parts are formulated as in \cite{bst:construct}; see also Rizell's work \cite{rizell:surgery}.

\begin{thm}[\cite{bst:construct, chantraine, ehk:leg-knot-lagr-cob, rizell:surgery}]
\label{thm:construct}
If two oriented Legendrian links $\leg_{-}$ and $\leg_{+}$ in the standard contact $\rr^{3}$ are related by any of the following three moves, then there exists an oriented exact Lagrangian cobordism  $\leg_{-} \prec_L \leg_{+}$.
\begin{description}
\item[Isotopy] $\leg_{-}$ and $\leg_{+}$ are Legendrian isotopic.
\item[$\boldsymbol{0}$-Handle] The front of $\leg_{+}$ is the same as that of $\leg_{-}$ except for the addition of a disjoint Legendrian unknot as in the top of Figure~\ref{fig:construct}.
\item[$\boldsymbol{1}$-Handle] The fronts of $\leg_{\pm}$ are related as in the bottom of Figure~\ref{fig:construct}.
\end{description}
\end{thm}

\begin{figure}
	\centerline{\includegraphics{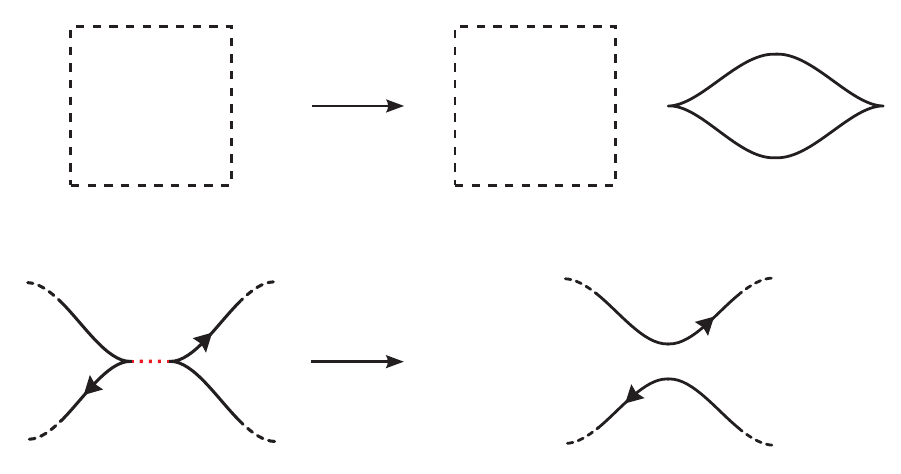}}
	\caption{The second and third moves of Theorem~\ref{thm:construct}.}
	\label{fig:construct}
\end{figure}

\begin{ex} \label{ex:torus-link}
Let $T_n$ denote the Legendrian $(2,n)$-torus link, where $n>0$ denotes the number of positive half-twists, shown in Figure~\ref{fig:(2,n)-torus}.  

	\begin{figure}
		\labellist
		\small
		\pinlabel $n-2$ [B] at 382 25
		\endlabellist
		\includegraphics[width=\linewidth]{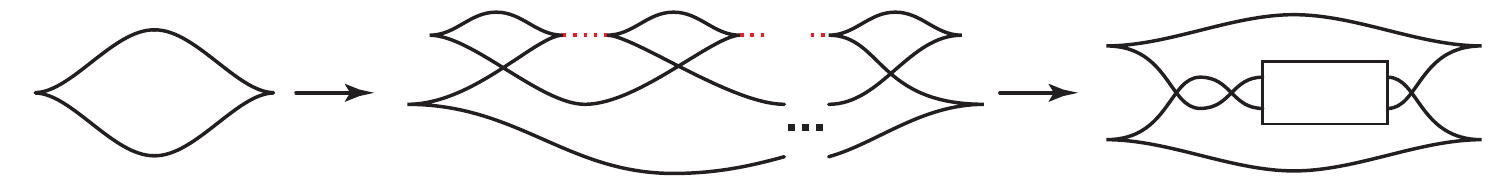}
		\caption{Construction of a filling of a $(2,n)$-torus link. }
		\label{fig:(2,n)-torus}
	\end{figure}

The link $T_n$ admits a decomposable filling with Euler characteristic $2-n$.  To see this, first create an unknot and then use the first Reidemeister move $n$ times on the top of the unknot as in Figure~\ref{fig:(2,n)-torus}.  Add $n - 1$ $1$-handles to obtain $T_n$. A simple calculation shows that $tb(T_n) = n - 2$, so by Equation~\ref{eq:tb-euler}, the Euler characteristic of the filling  is $2-n$. 
\end{ex}

In Section~\ref{ssec:high-d}, we will require Golovko's cobordism spinning construction \cite{golovko:tb} and a generalization of $1$-handle attachments to higher dimensions.  These constructions will be discussed as they are needed.
 
From now on, we will deliberately confuse the notation for a Legendrian link and  its front projection. In light of the above theorem, we say that an oriented Legendrian link has a \dfn{decomposable Lagrangian filling} if there exists a sequence of moves from Theorem~\ref{thm:construct} which takes $\emptyset$ to a front diagram of the link. It is an open question to determine if all Lagrangian fillings can be represented as decomposable fillings; see \cite{chantraine:collar} for further discussion.

\begin{rem} \label{rem:orientation}
	The orientations of the components of a link are important in the definition above.  For example, only one of the two orientations we may assign to the Hopf link allows for a decomposable filling.
\end{rem}

\section{Polyfillable Legendrian Links}
\label{sec:leg-proofs}

In this section, we prove Theorem~\ref{thm:leg-polyfillable}.  As mentioned in the introduction, we actually prove several more refined theorems that show that we can finely control the distribution of the topology of a filling across the components of a Legendrian link. 

\subsection{A First Example}
\label{ssec:first-ex}

We begin our discussion with a simple example of a polyfillable link $\leg^*$, which appears on the far right of Figure~\ref{fig:first example}.  Note that this example  is already enough to answer Question~\ref{ques:main}.  Constructions of two topologically distinct fillings of $\leg^*$ may be described as follows:

	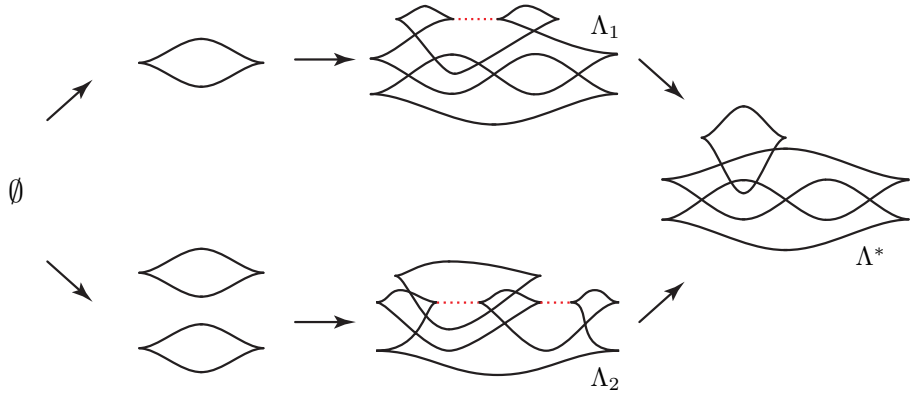
\begin{figure}
		\labellist
		\large
		\pinlabel $\emptyset$ [b] at 20 85
		\small
		\pinlabel $\leg_1$ [bl] at 273 160
		\pinlabel $\leg_2$ [tl] at 273 18
		\pinlabel $\leg^*$ [tl] at 391 72
		\endlabellist
		\centerline{\includegraphics[width=\linewidth]{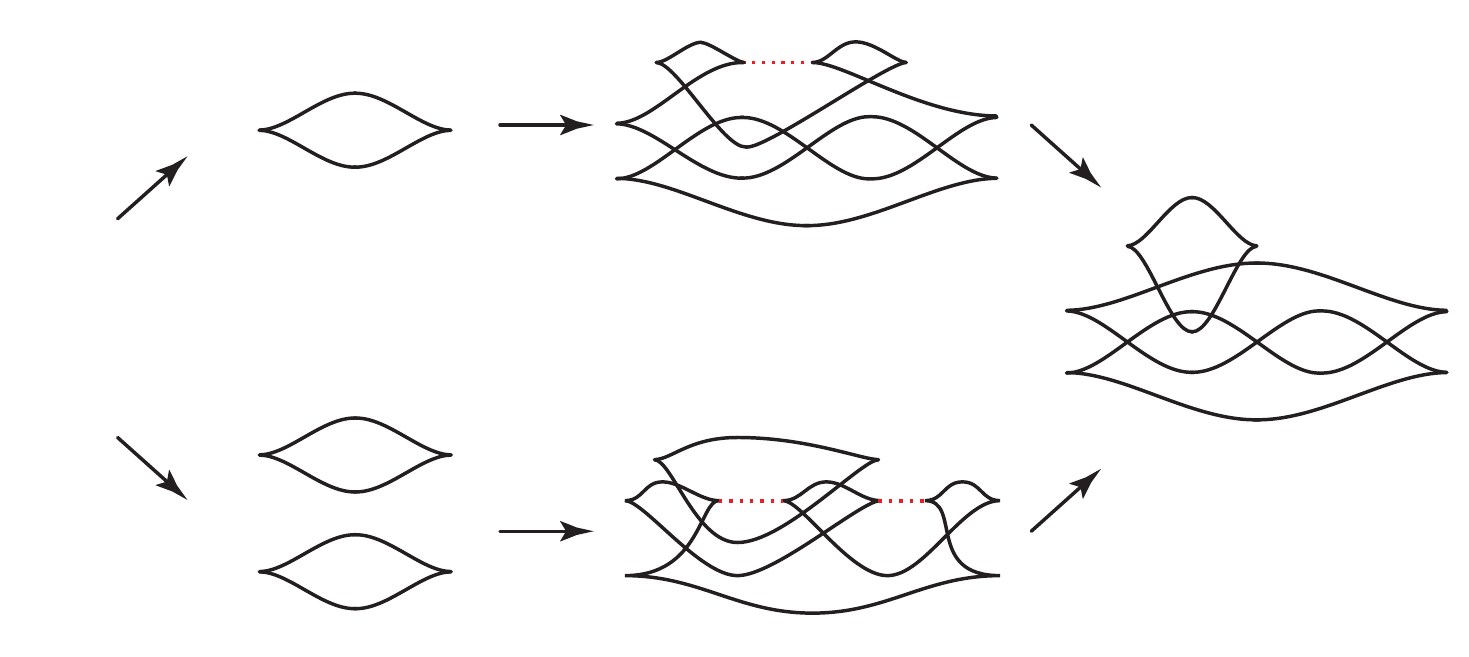}}
		\caption{The link $\leg^*$ on the right has two topologically distinct Lagranian fillings:  the top filling is an annulus, while the bottom is the union of a disk and a punctured torus.  } 
		\label{fig:first example}
	\end{figure}

\begin{description}	
\item[Construction A] Create a $0$-handle and obtain $\leg_1$ via Legendrian isotopy.  Attach a $1$-handle to obtain $\leg^*$. Notice that since there is but one $0$-handle, the filling created is a connected surface --- in particular, it is an annulus. 

\item[Construction B] Create two $0$-handles and obtain $\leg_2$ via Legendrian isotopy. Then add two $1$-handles to obtain $\leg^*$.  Since  no $1$-handle is attached between the two link components, the filling created is a disconnected surface; the top component of $\leg^{*}$ bounds a disk, while the bottom bounds a punctured torus. 
\end{description}

\subsection{The Tangle Replacement Construction}
\label{ssec:entangle}
	
The polyfillable link $\leg^*$ discussed above may be viewed as an entanglement of the unknot and the trefoil. The manner in which these link components are entangled may be generalized through the following tangle replacement procedure.

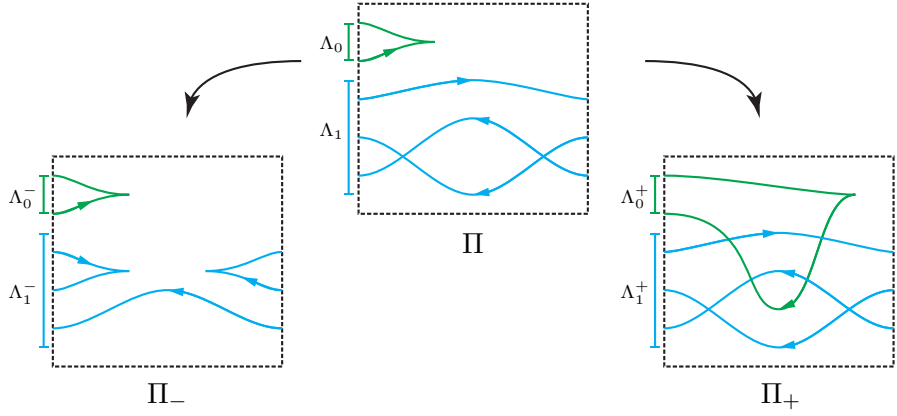
\begin{figure} 
	\labellist
		\pinlabel $\Pi_-$ [b] at 81 3
		\pinlabel $\Pi$ [b] at 225 77
		\pinlabel $\Pi_+$ [b] at 370 3
		\tiny
		\pinlabel $\leg_0$ [r] at 168 179
		\pinlabel $\leg_1$ [r] at 168 138
		\pinlabel $\leg_0^+$ [r] at 312 107		
		\pinlabel $\leg_1^+$ [r] at 312 62
		\pinlabel $\leg_0^-$ [r] at 24 107
		\pinlabel $\leg_1^-$ [r] at 24 62
		\endlabellist
		\centerline{\includegraphics[width=5in]{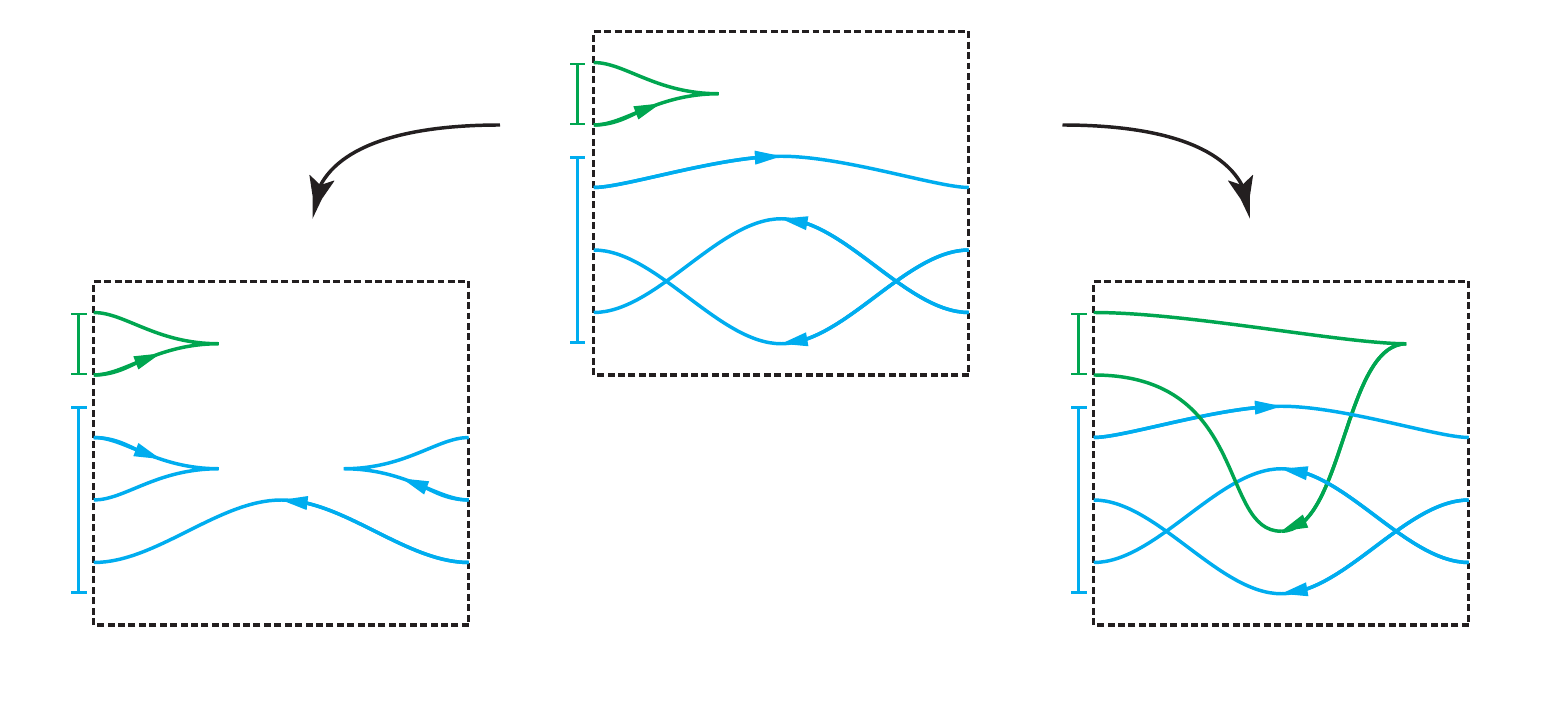} }
		\caption{The tangles $\Pi$, $\Pi_-$, and $\Pi_+$. } 
		\label{fig:pi-tangles}
	\end{figure} 	

\begin{defn} \label{defn:pi-tangle}
If a Legendrian link $\leg$ contains the tangle in the center of Figure~\ref{fig:pi-tangles} with the top two strands belonging to a component $\leg_0$ and the bottom three strands belonging to a sublink $\leg_1$ that does not contain $\leg_0$, then we say that $\leg$ \dfn{contains a $\Pi$ tangle}. 
Denote the tangle replacements on the left and right sides of Figure~\ref{fig:pi-tangles} by $\leg^-_{\Pi}$ and $\leg^+_{\Pi}$, respectively.
\end{defn}

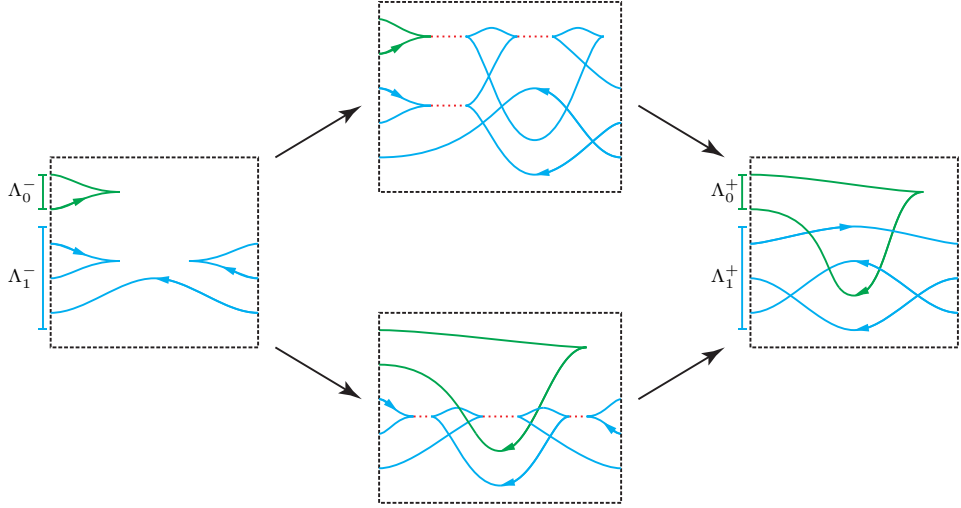
\begin{figure}
		\labellist
		\tiny
		\pinlabel $\leg_0^+$ [r] at 382 172		
		\pinlabel $\leg_1^+$ [r] at 382 127
		\pinlabel $\leg_0^-$ [r] at 16 172
		\pinlabel $\leg_1^-$ [r] at 16 127
		\endlabellist
		\centerline{\includegraphics[width=\linewidth]{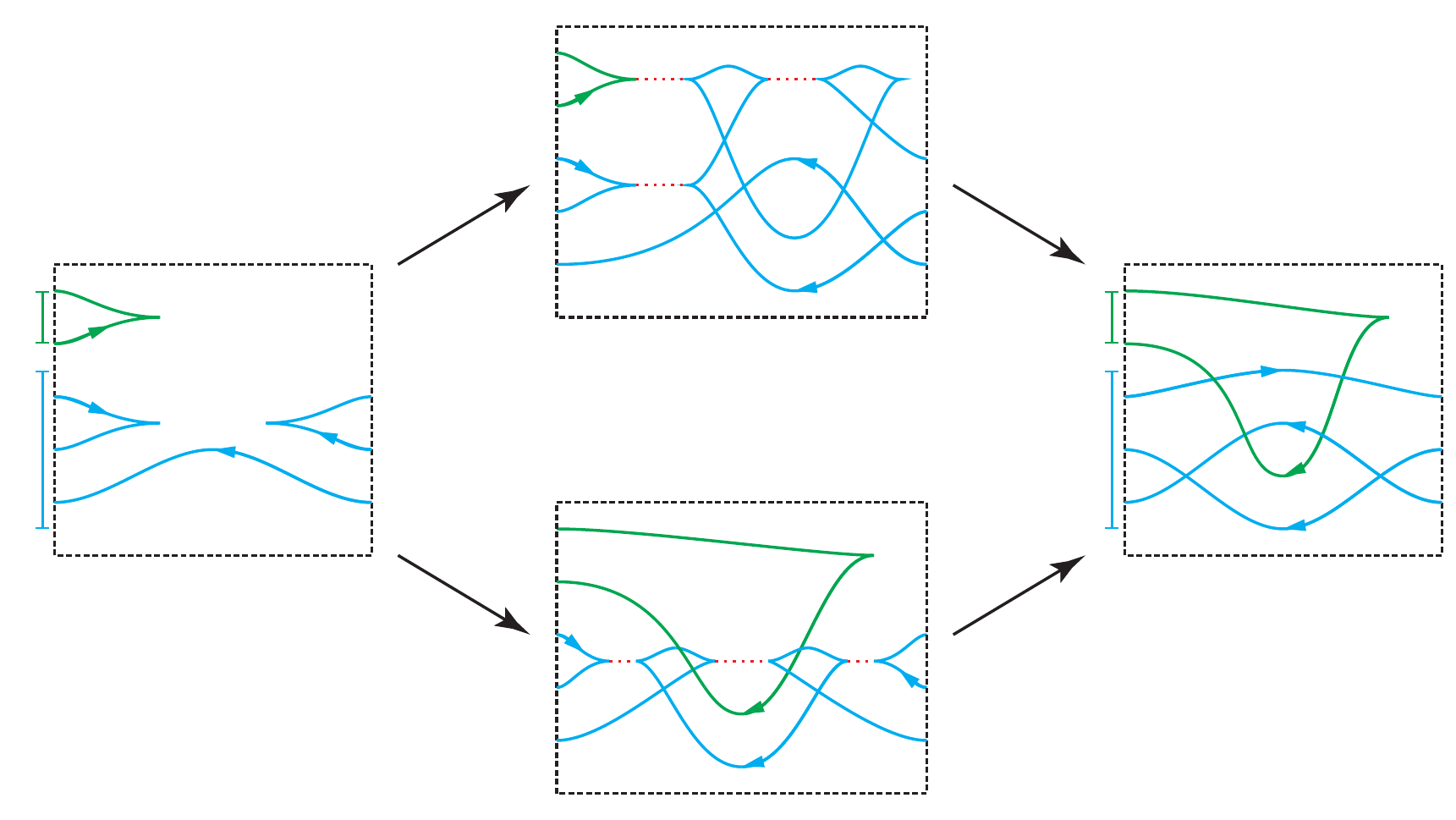}}
		\caption{The upper sequence defines the join construction and the lower sequence defines the split construction. In each case, Legendrian isotopy is implied between the first and second diagrams, and dotted horizontal bars indicate 1-handle attachments. } 
		\label{fig:local cobordisms}
\end{figure}

The top and bottom sequences in Figure~\ref{fig:local cobordisms} define topologically distinct cobordisms between $\leg^-_{\Pi}$ and $\leg^+_{\Pi}$. In the upper sequence, the \textbf{join construction}, $\leg^-_0$ and $\leg^-_{1}$ are merged by attaching a Lagrangian $1$-handle. It follows that $\leg^+_0$ and $\leg^+_1$ are contained in the boundary of a single connected component of the cobordism. In the lower sequence, the \textbf{split construction}, all 1-handles are attached within $\leg^-_1$. Thus, $\leg^+_0$ and $\leg^+_1$ belong to the boundaries of disjoint components.

The following theorem recasts the above discussion in a computationally useful form, using Equation~(\ref{eq:tb-euler}) to compute the changes in the Euler characteristics of the fillings.

\begin{thm} \label{thm: big theorem} Suppose the Legendrian link $\leg$ contains a $\Pi$ tangle involving the sublinks $\leg_0$ and $\leg_1$. If $\leg^-_{\Pi}$ has a Lagrangian filling $L^-$ with  $\leg^-_0$ and $\leg^-_{1}$ lying on the boundaries of different connected components $L^-_0$ and $L^-_1$ of $L^-$, then $\leg^+_{\Pi}$ has two topologically distinct fillings characterized as follows:
\begin{enumerate}
\item The join construction yields a filling $L^+$ of $\leg^+_{\Pi}$ with one fewer connected component than $L^-$.  The Euler characteristic of the component arising from the merging of two previously disconnected components is $\chi(L_0^-) + \chi(L^-_1) - 3$.
\item The split construction yields a filling $L^+$ of $\leg^+_{\Pi}$ with the same number of connected components as $L^-$.  The component of $L^+$ extending $L^-_0$ has the same Euler characteristic as $L^-_0$, while the component that extends $L^-_1$ has Euler characteristic $\chi(L_1^-) - 3$.
\end{enumerate}
\end{thm}

\subsection{An Infinite Family of Polyfillable Links}
\label{ssec:inf-fam}

We are now ready to construct a family of polyfillable links; Theorem \ref{thm:leg-polyfillable} will be a direct consequence.  The family will be indexed by compositions of an integer $N$, i.e.\  positive integer vectors  $X=(x_1, \ldots, x_n)$ such that $\sum_{i=1}^n x_i = N$.  We say that $n$ is the length of $X$.  Following \cite[p.\ 18]{stanley}, the set of compositions $C(N)$ of $N$ is in bijective correspondence with the power set of $\{1, \ldots, N-1\}$, which can be seen as follows:  represent $N$ as a line of $N$ dots, and represent a length $k$ composition of $N$ as a set of $k-1$ vertical bars that separate the dots into sets whose size is determined by the composition.  For example, the composition $(2,1,3)$ of $6$ looks like the following:
$$\begin{array}{cc|c|ccc}
	\bullet &\bullet & \bullet & \bullet& \bullet& \bullet
\end{array}$$
Thus, the set $C(N)$ inherits a natural partial order $\leq$ from the partial order on the power set given by inclusion.  From the perspective of dot-and-bar diagrams, we have $Y \leq X$ if the dot-and-bar diagram of $Y$ can be obtained from that of $X$ by removing zero or more bars. Numerically, $Y$ is the result of adding together disjoint strings of adjacent entries in $X$. Note that taking the set of compositions of $N$ modulo permutation of the components yields the set of partitions of $N$.

With this notation in hand, we may describe the family of polyfillable links and their fillings.

\begin{thm} \label{thm:general case}
Given a composition $X$ of an integer $N > 1$, where  $X$ has length $n>1$, there exists a polyfillable Legendrian link $\leg_X$ such that for any composition $Y \leq X$, there exists a Lagrangian filling $L_Y$ whose $i^{th}$ component has Euler characteristic $-y_i$.
\end{thm}

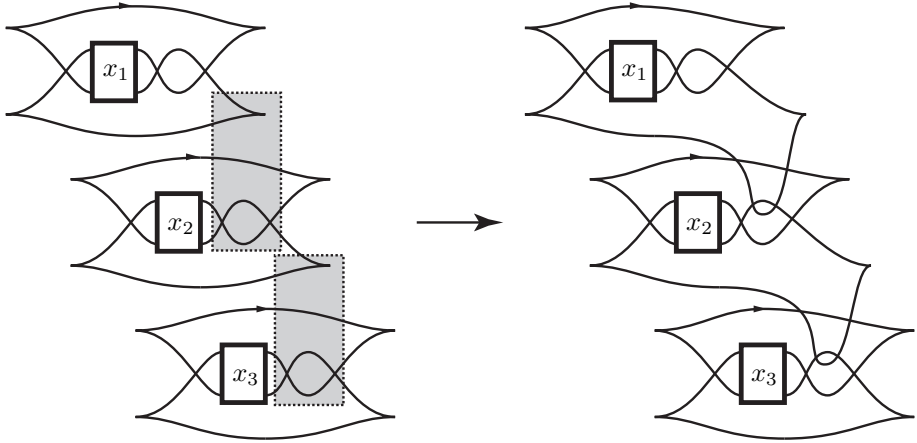
\begin{figure}
\begin{center}
		\labellist
		\small
		\pinlabel $x_3$ [B] at 109 33
		\pinlabel $x_2$ [B] at 82 96
		\pinlabel $x_1$ [B] at 55 160
		\pinlabel $x_1$ [B] at 271 160
		\pinlabel $x_2$ [B] at 298 96
		\pinlabel $x_3$ [B] at 325 33
		\endlabellist
\includegraphics[width=\linewidth]{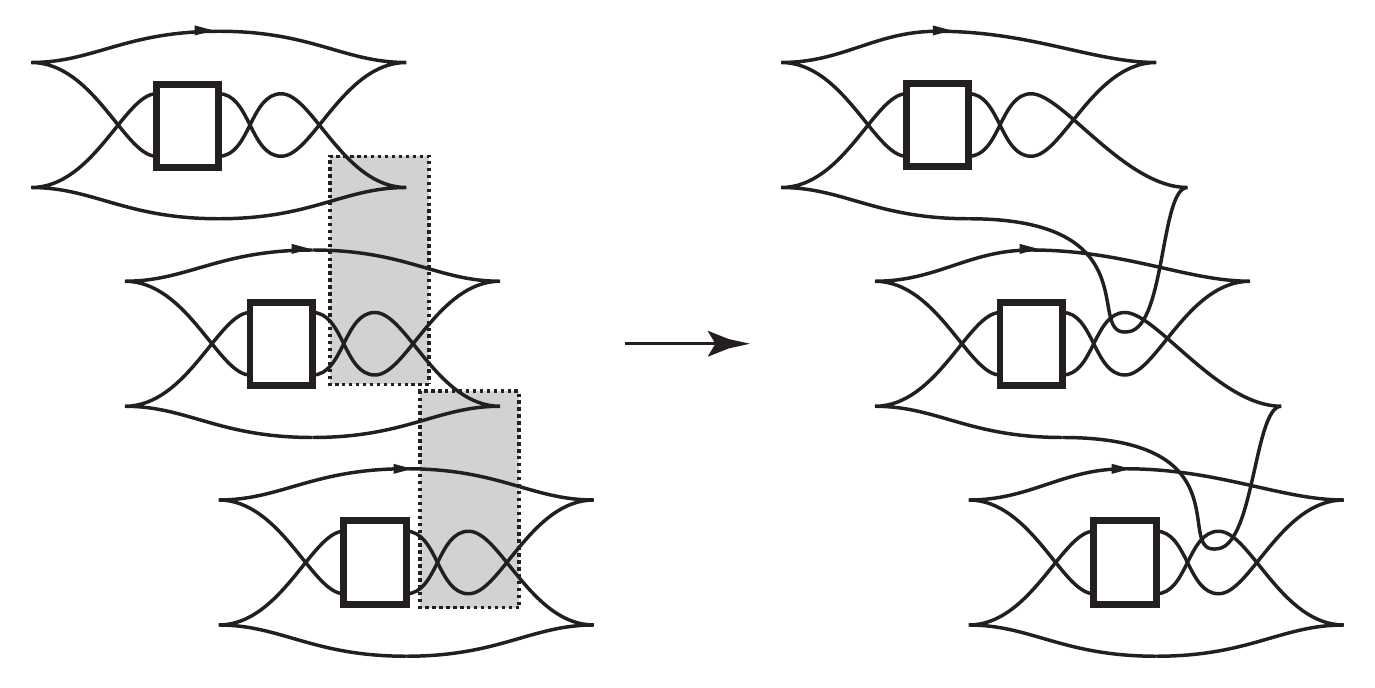}
\caption{The split link $\leg^s$ corresponding to a length-$3$ composition and its conversion to $\leg_X$ via several $\Pi$ tangle replacements.}
\label{fig:N-trefoil}
\end{center}
\end{figure}

\begin{proof} \label{prf:short}
Consider the split link $\leg^s=\bigsqcup_{i = 1}^n \leg_i$ depicted in Figure \ref{fig:N-trefoil}, where $\leg_i$ is a $(2,x_i+2)$-torus link.  Notice that the bottom-right cusp of $\leg_i$ and the two rightmost crossings of $\leg_{i+1}$ form a $\Pi$ tangle, giving rise to the $n-1$ $\Pi$ tangles shown in Figure \ref{fig:N-trefoil}. 

To create $\leg_X$, simply replace these $n-1$ $\Pi$ tangles by $\Pi_+$. To understand the fillings $L_Y$, suppose that $Y \leq X$; that is, suppose that 
\begin{equation*}
Y = (x_1 + \cdots + x_{i_1}, x_{i_1+1} + \cdots + x_{i_2}, \ldots, x_{i_k+1} + \cdots + x_n).
\end{equation*}

To create $L_Y$, first replace the aforementioned $\Pi$ tangles of $\leg$ by $\Pi_-$.  The resulting link is a split link whose components are unknots and $(2,m)$-torus knots, all of which are individually fillable by Example~\ref{ex:torus-link}.  Complete $L_Y$ by performing the split construction on the $\Pi$ tangles between the sublinks $\leg_{i_j}$ and $\leg_{i_j+1}$, $j=1, \ldots, k$, and the join construction on the $\Pi$ tangles between the other sublinks.   

To verify that we have, indeed, constructed the filling $L_Y$ described in the statement of the theorem, first note that Theorem~\ref{thm: big theorem}, used inductively, tells us that the resulting filling has $k+1$ components.  Further, the Euler characteristic of the $j^{th}$ component may be computed using Equation (\ref{eq:tb-euler}), which equates the Euler characteristic of each component to the negative of the Thurston-Bennequin invariant of its boundary.  Since the boundary of the $j^{th}$ component consists of $(2,x_j+2)$-torus knots linked together, and since one may easily see that the crossings between two knots contribute zero to the writhe of the diagram, we may conclude that the Euler characteristic of the $j^{th}$ component of $L_Y$ is $-\sum_{m=i_j+1}^{i_{j+1}} x_m$, as required.
\end{proof}

To derive Theorem~\ref{thm:leg-polyfillable} from Theorem~\ref{thm:general case}, simply consider the composition $X = (1,1,\ldots, 1)$ of $N$.  Note that the link $\leg_X$ in this case is necessarily composed of trefoil components.  The composition $X$ is maximal with respect to the partial order, and hence the number of topologically distinct fillings of $\leg_X$ is equal to the number of \emph{partitions} of $N$, as the topological type of the filling is insensitive to order.

Looking at the increasing sequence
\begin{equation*}
	(N) \leq (1,N-1) \leq (1,1,N-2) \leq \cdots \leq (1,1, \ldots, 1)
\end{equation*}
of compositions of $N$, we obtain the following corollary:

\begin{cor} \label{cor:polyfillable-diff-num-comp}
For every positive integer $N$, there exists an $N$-component Legendrian link $\leg_N \subset (S^3, \xi_0)$ such that for any $1 \leq k \leq N$, there exists a Lagrangian filling $L_{k,N} \subset (B^4, \omega_0)$ of $\leg_N$ with $k$ connected components.
\end{cor}

\section{Applications of Theorem \ref{thm:leg-polyfillable}}
\label{sec:applications}

The goal of this section is to derive Theorems
\ref{thm:alg-polyfillable},  \ref{thm:cont-polyfillable}, and \ref{thm:high-d-leg-polyfillable} from Theorem~\ref{thm:leg-polyfillable}.

\subsection{Symplectic Fillings of Transverse Links}
\label{ssec:transverse}

The first step in the proofs of Theorems \ref{thm:alg-polyfillable} and \ref{thm:cont-polyfillable} is to $C^\infty$ approximate a Lagrangian filling of a Legendrian link by a symplectic filling of a transverse link.

\begin{lem} \label{lem:perturb}
  Let $(X,L, \omega)$ be an oriented strong Lagrangian filling of $(Y,\leg, \alpha)$.  The Lagrangian surface $L$ may be $C^\infty$ approximated by a symplectic surface $L^+$ that satisfies:
  \begin{enumerate}
  \item $\omega|_{L^+} > 0$ and
  \item $\partial L^+$ is a positive transverse link smoothly isotopic to $\leg$.
  \end{enumerate}
\end{lem}

\begin{proof}
	By Lemma 2.3.A in Eliashberg's paper \cite{yasha:lagr-cyl}, it suffices to perturb $L$ in a collar neighborhood of $\partial L$ so that the boundary of the resulting surface is a transverse link and so that $\omega$ is positive on a neighborhood of the boundary of the perturbed surface.  

	We work in local coordinates in a collar neighborhood $(0,1] \times S^1 \times \rr^2$ of the boundary of $L$, where the coordinates are $(s,\theta,x,y)$ and the symplectic form is $\omega = d\bigl(e^s (\cos \theta \,dx - \sin \theta\, dy) \bigr)$. The Lagrangian $L$ is the  image of the embedding
	\begin{align*}
	\phi: \rr \times S^1 &\to \rr \times S^1 \times \rr^2 \\
	(s,t) &\mapsto (s,t,0,0),
	\end{align*}
Orient $L$ by specifying that $\rr \times S^1$ is oriented by $ds \wedge dt$.

	To perturb $L$, we choose a smooth increasing function $\sigma: (0,1] \to \rr$ that is equal to $0$ near $0$ and is equal to $\epsilon>0$ near $1$.  Define the perturbed surface using the embedding $\phi_+$ given by the following formula:
	$$\phi_+(s,t) = (s, t, \sigma (s) \sin t, \sigma(s) \cos t).$$
	It is clear from the definition of $\sigma$ that this perturbation may be made as small as desired and can be extended to the identity outside of the given neighborhood, so it suffices to show that $\phi_+^* \omega$ is a non-negative multiple of $ds \wedge dt$.  This is a straightforward computation:
	\begin{equation*}
		\phi_+^* \omega =  e^s( \sigma(s) + \sigma'(s) ) ds \wedge dt.
	\end{equation*}
	Since $e^s( \sigma(s) + \sigma'(s) ) \geq 0$, the conclusion follows.
	\end{proof}

\subsection{Algebraically Polyfillable Links}
\label{ssec:alg-polyfillable}

In this subsection, we prove Theorem~\ref{thm:alg-polyfillable}, namely that for every  integer $N>1$, there exists a link $K \subset S^{3}$ with $p(N)$ non-singular complex algebraic curves that intersect $S^3$ transversally in $K$ with the property that the intersections of the curves with $B^4$ are pairwise non-homeomorphic. Let $\leg$ be the Legendrian link constructed for Theorem~\ref{thm:leg-polyfillable} and let $L_{Y}$ be the Lagrangian filling corresponding to the composition $Y$ of $N$.  By Lemma~\ref{lem:perturb}, we may perturb $L_{Y}$ to a symplectic surface $L^{+}_{Y}$ in $B^{4}$ whose boundary is a positive transverse link $K$.  By Theorem 2 of \cite{bo:qp}, the surface $L^{+}_{Y}$ is quasi-positive.  Finally, Rudolph proved in \cite{rudolph:qp-alg} (as interpreted in \cite[\S4]{rudolph:braided-surface}) that every quasi-positive surface $S$ is isotopic to a smooth piece of an algebraic curve inside $B^4$ that intersects $S^3$ transversally in a link isotopic to $\partial S$.  This completes the proof.

\subsection{Polyfillable Contact $3$-Manifolds}
\label{ssec:cont-polyfillable}

We now pivot from fillings of links to fillings of contact $3$-manifolds.

Fix $N>1$ and let $L_{1}, \ldots, L_{N}$ be the Lagrangian fillings of the Legendrian link $\leg_N$ provided by Corollary~\ref{cor:polyfillable-diff-num-comp}. We will abuse notation and assume that the surfaces $L_{k}$ have already been perturbed, as in Lemma~\ref{lem:perturb}, to symplectic surfaces with transverse boundary.  Define $(X_{k}, \omega_{k})$ to be the double branched cover of $B^4$, branched over $L_{k}$, and let $(Y, \alpha)$ be the double branched cover of $S^3$, branched over $\leg_N$.  The construction of the symplectic form on the double branched cover may be found in \cite{gompf:sympl-asph}, for example. Since we know that the boundaries of the symplectic surfaces $L_{k}$ have collar neighborhoods and that $B^4$ has a collar neighborhood symplectomorphic to a portion of the symplectization of $S^3$, we know that the same is true for $X_{k}$ and $Y$.  It follows that $(X_{k}, \omega_{k})$ is a strong filling of $(Y, \alpha)$. Alternatively, as pointed out by Loi and Piergallini \cite[Theorem 3]{lp:stein}, the work in the previous subsection actually allows us to construct an analytic double branched cover over an algebraic curve; in this case, we see that $(X_{k}, \omega_{k})$ is, in fact, a Stein filling.

To finish the proof of Theorem~\ref{thm:cont-polyfillable}, then, we need only prove the following lemma:

\begin{lem} \label{lem:h3-calc}
  The rank of $H_3(X_{k}, Y)$ is $k-1$, but all $X_{k}$ have the same Euler characteristic.
\end{lem}

\begin{proof}
  The key tool in this proof is the following long exact sequence, adapted from \cite[Theorem 1]{lee-weintraub}:\footnote{While the long exact sequence in \cite{lee-weintraub} is only discussed for $\zz_2$ coefficients, it is clear that if both the base and the branched cover are orientable, as they are here, the proof carries over to $\zz$ coefficients.}

\begin{equation} \label{eq:2x-les}
  \cdots \to H_*(B^4, L_{k} \cup S^3) \to H_*(X_{k},Y) \to  H_*(B^4, S^3) \to \cdots.
\end{equation}

The long exact sequence shows that the Euler characteristic $\chi(X_{k},Y)$ may be computed from the Euler characteristics $\chi(B^4, S^3)$ (which is simply $1$) and $\chi(B^4, L_{k} \cup S^3)$.  To compute the latter quantity, we use the long exact sequence of the triple $(B^4, L_{k} \cup S^3, S^3)$, which shows that $\chi(B^4, L_{k} \cup S^3)$ depends only on $\chi(L_{k} \cup S^3, S^3)$.  Excision shows that 
$$\chi(L_{k} \cup S^3, S^3) = \chi(L_{k},\leg_N),$$
which does not depend on $k$ by Chantraine's theorem (\ref{eq:tb-euler}).  Thus, $\chi(X_{k},Y)$ does not depend on $k$.

On the other hand, the rank of $H_3(X_{k}, Y)$ does depend on $k$.  Theorem 2 of \cite{lee-weintraub} implies that the connecting homomorphism $$\partial_*:  H_4(B^4, S^3) \to H_3(B^4, L_{k} \cup S^3)$$ is injective, and hence it suffices to compute $H_3(B^4, L_{k} \cup S^3)$.  Since $H_3(B^4, S^3)$ and $H_2(B^4, S^3)$ obviously vanish, the long exact sequence of the triple $(B^4, L_{k} \cup S^3, S^3)$ and excision yield the isomorphisms
\begin{equation} \label{eq:les+exc}
  H_3(B^4, L_{k} \cup S^3) \simeq H_2(L_{k} \cup S^3, S^3) \simeq H_2(L_{k}, \leg_N).
\end{equation}

Since $L_{k}$ has $k$ components, the rank of $H_2(L_{k}, \leg_N)$ is $k$.  The lemma now follows from (\ref{eq:2x-les}) and (\ref{eq:les+exc}).
\end{proof}

\subsection{Higher-Dimensional Fillings}
\label{ssec:high-d}

We finish our exploration of the applications of Theorem~\ref{thm:leg-polyfillable} by examining its generalization to higher dimensions and the consequences of the generalization for the generating family homology and linearized Legendrian contact homology invariants. 

\subsubsection{Connected Polyfillable Legendrian Submanifolds}

To generalize Theorem~\ref{thm:leg-polyfillable} to higher dimensions, we need two further constructions of Lagrangian cobordisms.  The first is \dfn{spinning}.  Suppose that a front diagram for a Legendrian submanifold $\leg \subset \rr^{2n+1}$ lies in the half-space $x_1>0$.  The front spinning construction of Ekholm, Etnyre, and Sullivan \cite{ees:high-d-geometry} produces a new Legendrian submanifold $\Sigma \leg \subset \rr^{2n+3}$ that is homeomorphic to $\leg \times S^1$.  Golovko \cite{golovko:tb} extended this construction to Lagrangian cobordisms:  a Lagrangian cobordism $\leg_- \prec_L \leg_+$ lying in the $x_1>0$ half space may be spun to a Lagrangian cobordism $\Sigma \leg_- \prec_{\Sigma L} \Sigma \leg_+$.

The second construction is a generalization of the $1$-handle attachment in Theorem~\ref{thm:construct} to Legendrian surfaces in $\rr^5$.  As formulated in \cite{bst:construct}, but see also \cite{rizell:surgery}, if two Legendrian surfaces $\leg_-$ and $\leg_+$ in the standard contact $\rr^5$ are related as in Figure~\ref{fig:2d-1-handle}, then there is an oriented Lagrangian cobordism $\leg_- \prec_L \leg_+$.  The cobordism consists of attaching a Lagrangian $1$-handle to $(-\infty, -T] \times \leg_- \subset \rr \times \rr^5$.

\begin{figure}
\begin{center}
\includegraphics[width=.9\linewidth]{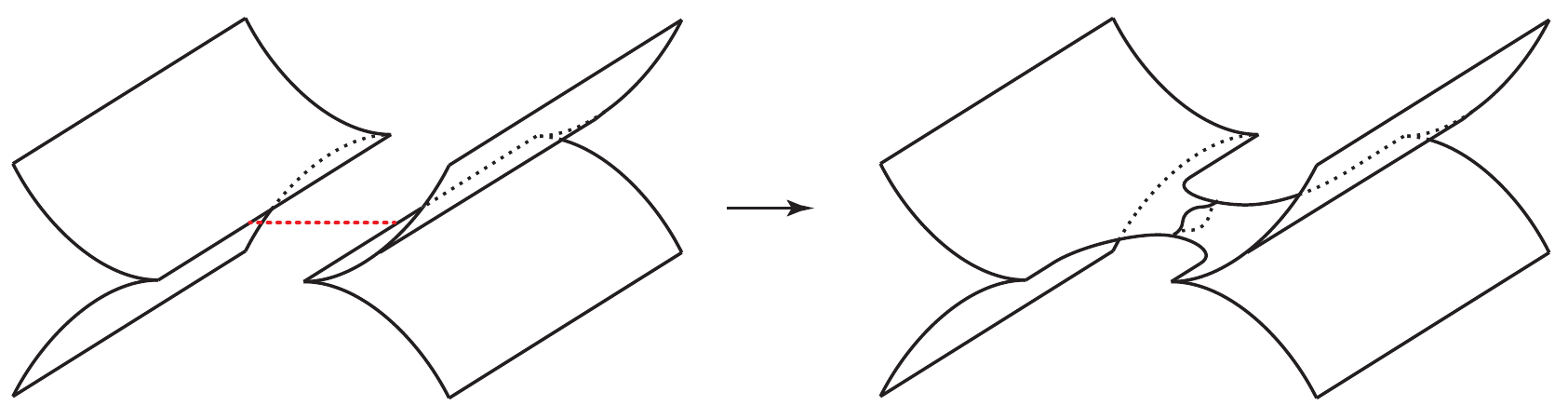}
\caption{Attaching a Lagrangian $1$-handle to a Legendrian surface $\leg_-$.  The resulting Lagrangian cobordism is orientable if the orientations of $\leg_-$ are opposite on either side of the dotted line.}
\label{fig:2d-1-handle}
\end{center}
\end{figure}

We now have the tools to prove Theorem~\ref{thm:high-d-leg-polyfillable}.

\begin{proof}[Proof of Theorem~\ref{thm:high-d-leg-polyfillable}]
	We begin with the link $\leg_N$ of Corollary~\ref{cor:polyfillable-diff-num-comp}, whose components are all trefoil knots.  Note that $tb(\leg_N) = N$.  We make a slight modification to $\leg_N$ using one Reidemeister I move for all but the bottom-most component; see Figure~\ref{fig:spin-trefoil}. By Theorem~\ref{thm:construct}, this modification does not change the fillability properties of the link.  If a filling $L_k$ of $\leg_N$ has $k$ connected components, then Equation~(\ref{eq:tb-euler}) implies that 
	\begin{equation} \label{eq:dim-L}
		\rk H_1(L_k) = N+k.
	\end{equation}
	
\begin{figure}
\begin{center}
\includegraphics{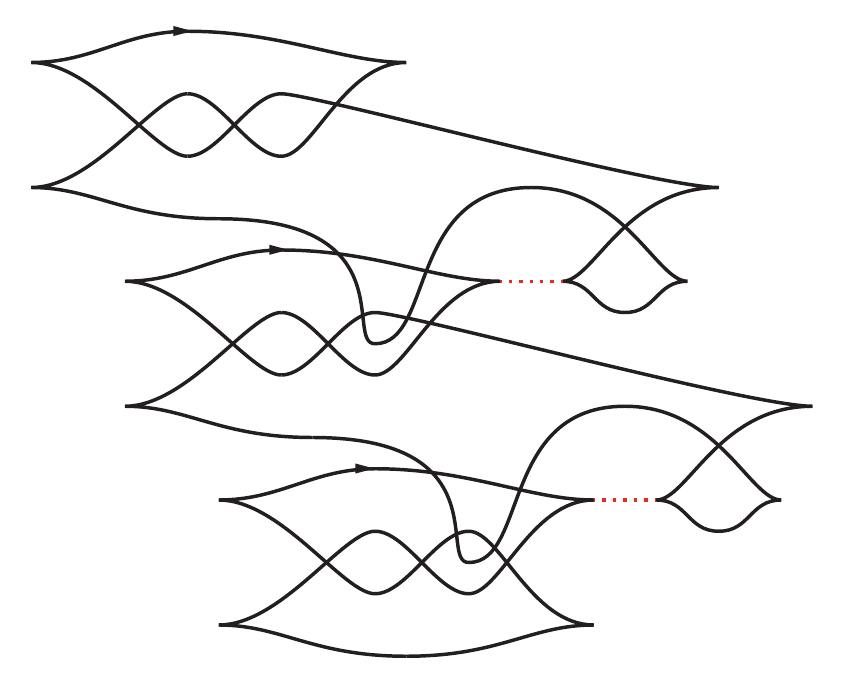}
\caption{Spinning the pictured link and then performing $1$-handle attachments at the dotted lines yields a connected polyfillable Legendrian surface.}
\label{fig:spin-trefoil}
\end{center}
\end{figure}
	
	Spin $\leg_N$ to a Legendrian surface $\Sigma \leg_N$ in $\rr^5$, and spin the fillings $L_k$ to fillings $\Sigma L_k$ of $\Sigma \leg_N$.  The K\"unneth theorem and Equation (\ref{eq:dim-L}) imply that
	\begin{equation} \label{eq:dim-spin-L-1}
		\rk H_1(\Sigma L_k) = N+2k.
	\end{equation}
	
	Finally, attach $N-1$ Lagrangian $1$-handles at the positions indicated in Figure~\ref{fig:spin-trefoil} to obtain Lagrangian fillings $\Sigma'L_k$ of a new Legendrian surface $\leg'_N$.  The resulting Legendrian is connected, as are all of the resulting Lagrangian fillings.  The fillings are still topologically distinct, however, as we may compute
	\begin{equation} \label{eq:dim-spin-L-2}
		\rk H_1(\Sigma' L_k) = 2k + 1.
	\end{equation}

	Finally, we may repeatedly spin the resulting Legendrian and its fillings to extend the result to Legendrian submanifolds of arbitrarily high dimension.
\end{proof}

\begin{rem}
	The Legendrian submanifolds constructed in the proof above are the product of a genus $N$ surface with an $(n-2)$-dimensional torus.  It is possible to use either a refinement of the spinning construction for fronts and cobordisms with boundary in \cite{bst:construct} or a sequence of higher embedded surgeries as in \cite{bst:construct, rizell:surgery} to produce polyfillable Legendrian \emph{spheres} of dimension $n>1$.
\end{rem}

\subsubsection{Chekanov Numbers in Higher Dimensions} \label{sssec:high-d}

Theorem~\ref{thm:high-d-leg-polyfillable} has implications for the structure of non-classical Legendrian invariants.  In this subsection only, we assume familiarity with generating family invariants of Legendrian submanifolds, as explained in \cite{f-r, lisa-jill,josh-lisa:obstr, lisa:links}, and with Legendrian contact homology and its linearizations, as explained in \cite{chv, ees:high-d-geometry}.  The linearized contact homology invariant is a collection of graded (homology) groups parametrized by augmentations of the Legendrian contact homology differential graded algebra. Similarly, the generating family invariant of a Legendrian $\leg$ takes the form of a collection of graded (homology) groups, parametrized by equivalence classes of generating families for $\leg$.    The cardinality of these sets is, itself, a Legendrian invariant, dubbed the \dfn{Chekanov number} (resp.\ \dfn{GF number}) of a Legendrian submanifold.  These numbers are known to be equal for links in $\rr^3$ \cite{f-r} and are conjecturally equal in higher dimensions.

An early question in the theory of Legendrian contact homology, proposed by Ng \cite{lenny:computable} and resolved for Legendrian knots in $\rr^3$ by Melvin and Shrestha \cite{melvin-shrestha} (see also Sivek \cite{sivek:bordered-dga}), was whether the Chekanov number could ever be greater than $1$.  For Legendrian knots in $\rr^3$, the answer is yes; in fact, the Chekanov number can be arbitrarily high.  

Theorem~\ref{thm:high-d-leg-polyfillable} may be used to produce examples of Legendrian submanifolds of $\rr^{2n+1}$ with arbitrarily high Chekanov / GF numbers.

\begin{prop} \label{prop:high-d-chv-num}
	For any integers $n>1$ and $N>0$, there exists a connected Legendrian submanifold $\leg \subset \rr^{2n+1}$ with Chekanov / GF number at least $N$.
\end{prop}

\begin{proof}
	We first prove the proposition for the GF number.	The proof relies on the fact, proved in \cite{josh-lisa:obstr}, that if a Lagrangian filling $L$ of a Legendrian submanifold $\leg$ is ``compatible'' with a generating family $f$ of $\leg$, then:
	\begin{equation} \label{eq:filling-iso}
		GH^*(f) \simeq H^{*+1}(L,\leg).
	\end{equation}
	Thus, if the Legendrian $\leg$ has $N$ compatible Lagrangian fillings, all of which have different relative cohomology, then the GF number of $\leg$ is at least $N$.  
	
	It suffices to prove, then, that the Lagrangian fillings $L_k$ constructed in the previous subsection are all compatible with a generating family of $\leg$. This, however, is a straightforward affair given the results of \cite{bst:construct}:  the isotopy, handle attachment, and spinning constructions discussed in Theorem~\ref{thm:construct} and  Section~\ref{sssec:high-d} all produce compatible Lagrangian cobordisms, so long as the $1$-handle attachments join cusps that represent births and deaths of fiber critical points of the generating family with the same Morse index; this  is easy to arrange.

	The proof of the proposition for the Chekanov number follows the same argument, with Equation~(\ref{eq:filling-iso}) replaced by a similar result that was stated and outlined by Ekholm \cite{ekholm:lagr-cob} and proven by Rizell \cite{rizell:lifting}.
\end{proof}

\begin{rem}
	Examples that prove the proposition above could also be constructed by spinning the examples of Melvin-Shrestha or of Sivek and using the techniques of Ekholm and K\'alm\'an \cite{ek:isotopy-tori} to compute the Legendrian contact homology. On the other hand, for each of Melvin-ShresthaÕs or SivekÕs examples, only one of the linearized contact or generating family homologies arises from a filling (as in Equation~(\ref{eq:filling-iso}), whereas all of the linearized contact or generating family homologies constructed in the proof above arise from Lagrangian fillings.
\end{rem}

\bibliographystyle{amsplain} 
\bibliography{main}

\end{document}